\documentclass[11pt,leqno,amscd,amssymb,verbatim, url]{amsart}
\usepackage{amsfonts,amssymb,graphicx}
\usepackage{amsmath,amscd}
\newtheorem{thm}{Theorem}[section]
\usepackage{amsfonts,latexsym, curves}

\newtheorem{cor}[thm]{Corollary}

\newtheorem{prop}[thm]{Proposition}

\theoremstyle{definition}

\newtheorem{defn}[thm]{Definition}

\newtheorem{conj}[thm]{Conjecture}

\newtheorem{rem}[thm]{Remark}
\newcommand{\bpict}{\begin{picture}}

\newcommand{\epict}{\end{picture}}

\newtheorem{rems}[thm]{Remarks}

\numberwithin{equation}{thm}

\newcommand{\grExt}{\text{\rm ext}}
\newcommand{\grHom}{\text{\rm hom}}

\newcommand{\wrDelta}{{\widetilde{\Delta}}^{\text{\rm red}}}

\newcommand{\wX}{{\widetilde{X}}}

\newcommand{\wrnabla}{\widetilde\rnabla}

\newcommand{\wOmega}{{\widetilde\Omega}}

\newcommand{\gr}{\text{\rm gr}}

\newcommand{\sC}{{\mathcal {C}}}

\newcommand{\wgr}{\widetilde{\text{\rm gr}}}

\newcommand{\Ext}{{\text{\rm Ext}}}

\newcommand{\fa}{{\mathfrak a}}

\newcommand{\Amod}{A{\text{\rm --mod}}}
\newcommand{\Bmod}{B\text{\rm --mod}}

\newcommand{\Agrmod}{A{\text{\rm--grmod}}}

\newcommand{\Bgrmod}{B{\text{\rm--grmod}}}
\newcommand{\Hom}{\text{\rm Hom}}

\newcommand{\End}{\operatorname{End}}
\newcommand{\op}{{\text{\rm op}}}

\newcommand{\wfa}{\widetilde{\mathfrak a}}

\newcommand{\sO}{{\mathcal{O}}}

\newcommand{\rad}{\operatorname{rad}}

\newcommand{\Dist}{\operatorname{Dist}}

\newcommand{\rDelta}{\Delta^{\text{\rm red}}}
\newcommand{\rnabla}{\nabla_{\text{\rm red}}}
\newcommand{\wM}{{\widetilde{M}}}
\newcommand{\wA}{{\widetilde{A}}}
\newcommand{\wXi}{{\widetilde \Xi}}

\newcommand{\wU}{{\widetilde U}}
\newcommand{\wDelta}{{\widetilde{\Delta}}}

\newcommand{\wP}{{finite dimensional\widetilde{P}}}

\newcommand{\Jan}{\Gamma_{\text{\rm Jan}}}

\newcommand{\Resreg}{\Gamma_{\text{\rm res,reg}}}

\newcommand{\blist}{\begin{list}{\rom{(\roman{enumi})}}{\setlength
{\leftmarg in}{0em} \setlength{\itemindent}{7ex}
\setlength{\labelsep}{2ex}\setlength{\listparindent}{\parindent}
\usecounter{enumi}}}
\newcommand{\elist}{\end{list}}

\newcommand{\grAmod}{{A{\text{\rm -grmod}}}}

\dedicatory{We dedicate this paper to the memory of Julie Riddleberger.}

\begin{document}

 \title[From forced gradings to Q-Koszul algebras]{From forced gradings to Q-Koszul algebras}
\author{Brian J. Parshall}\address{Department of Mathematics \\
University of Virginia\\
Charlottesville, VA 22903} \email{bjp8w@virginia.edu {\text{\rm
(Parshall)}}}
\author{Leonard L. Scott}
\address{Department of Mathematics \\
University of Virginia\\
Charlottesville, VA 22903} \email{lls2l@virginia.edu {\text{\rm
(Scott)}}}

\thanks{Research supported in part by the National Science
Foundation}


\medskip
\medskip
\medskip \maketitle

\begin{abstract} 
This paper has two parts. Part I  surveys some recent work by the authors in which ``forced'' grading constructions have played a significant role in the representation theory of semisimple algebraic groups $G$ in positive characteristic.
The constructions begin with natural finite dimensional quotients of the distribution algebra ${\text{\rm Dist}}(G)$, but then ``force'' gradings into the picture by passing to positively graded algebras constructed from certain ideal filtrations of these quotients. This process first guaranteed a
place for itself by proving, for large primes, that all Weyl modules have $p$-Weyl filtrations. Later it led, under similar circumstances, to a new ``good filtration" result for restricted Lie algebra Ext groups between $p$-restricted irreducible $G$-modules. In the process of proving these results, a new kind of graded algebra was invented, called a Q-Koszul algebra. Recent conjectures suggest these algebras arise in forced grading constructions as above, from quotients of $\Dist(G)$, even for small primes and even in settings involving singular weights. Related conjectures point to a promising future for using Kazhdan-Lusztig theory to relate quantum and algebraic group cohomology and Ext groups in these same small prime and singular weight settings. One of these conjectures is partly proved in Part II of this paper.
The proof is introduced by remarks of general interest on positively graded algebras and Morita equivalence, followed
by a discussion of recent Koszulity results of Shan-Varagnolo-Vasserot, observing some extensions of their work.\end{abstract}

\vskip.2in
\begin{center}{\Large\bf Introduction}\end{center}
\vskip.2in

Many questions in the modular representation theory of a semisimple algebraic group $G$ reduce to questions
involving certain families of finite dimensional algebras $A$ which capture part of the story for $G$. This applies both to structural
questions of specific classes of modules, but also to the cohomology theory of $G$. In addition, the existence
of a suitable positive grading on $A$ would provide a powerful tool to attacking many natural questions. Unlike
the case of the BGG category $\sO$, such positive gradings are difficult (and maybe impossible) to obtain. Several years ago
the authors discovered a way to construct (or ``force") from $A$ an associated positively graded algebra $\wgr A$ which, remarkably, enjoys
many nice properties. For example, it can often be shown that $\wgr A$ is a quasi-hereditary algebra. In addition, the
authors conjecture that $\wgr A$ is a standard Q-Koszul algebra, a new class of algebras which stand between
quasi-hereditary and Koszul algebras. This conjecture is made for almost all characteristics $p$ of the base
field of $G$. To date, we have verified the conjecture for $p$ very large, but also in at least one highly non-trivial
case when $p=2$. Taken together with several companion conjectures, these results suggest a new global picture
for the representation/cohomology theory of $G$ which is largely independent of the validity of the Lusztig character formula.
It is ``global" for at least two reasons: first, it should hold for ``small" characteristics, and, second, it is not bound
by any Jantzen region, applying instead for any arbitrarily large finite saturated set of dominant weights.

This paper is organized as follows. Part I surveys the authors' work over the past five years on  forced graded
methods and results. These lead up to the definition of Q-Koszul and standard Q-Koszul algebras in \S4, where the
main conjecture mentioned above is stated. (The notions of Q-Koszul and standard Q-Koszul algebras first appear in 
\cite{PS13b}, \cite{PS13d}, and then in \cite{PS14}, which first states the conjecture and several others.)  Part II contains a number of new
results. For example, in \S5, it is proved that the Q-Koszul property for a finite dimensional algebra is invariant
under Morita equivalence. In \S6, we discuss a recent result of Shan-Varagnolo-Vasserot \cite{SVV14} on the
Koszulity of $q$-Schur algebras, focusing on their general ``parabolic-singular duality" theory for affine Lie algebras. Then we indicate how the theory may be used to treat Koszulity, with some restrictions, for other
quantum types (i.e., other than type $A$). With this result
in hand, \S7 returns to the conjectures given in \cite{PS14}, proving part of one of them.

\medskip
   The authors dedicate this paper to the memory of Julie Riddleberger. For decades, she held an iconic position on our departmental
staff. Without visible complaint,  she typed or  assisted in the typing, of many mathematical papers, often catching mathematical errors
and, sometimes, even correcting them as she worked.

\vskip.3in
\begin{center}{\bf Notation}\end{center}
We generally follow Jantzen's book \cite{JanB} for basic material on the representation theory of a reductive algebraic
group $G$. A notable exception: the root system of $G$ is denoted $\Phi$, and $\Pi$ is a fixed set of simple roots  (rather than $R$ and $S$ as in \cite{JanB}). Nevertheless, here
are some standard notations to be used:

\begin{itemize} 

\item[(1)] Given a finite root system $\Phi$, let $m$ be (temporarily) the maximum of all heights of roots $\alpha\in\Phi$ (taken with respect to some fixed set of
simple roots). Then $h:=m+1$ is the Coxeter number of $\Phi$. The Coxeter number of a reductive group $G$  is defined to be the Coxeter number of its root system $\Phi.$.

\item[(2)] Let $G$ be a simply connected reductive algebraic group over an algebraically closed field $k$. Let $T$ be a fixed maximal torus and let $X=X(T)$ be the weight lattice of $T$. Fix a positive Borel subgroup, and let $X_+=X(T)_+$ be the
corresponding set of dominant weights.  Recall $\lambda\in X$ is $p$-regular
if $(\lambda+\rho,\alpha^\vee)\not\equiv 0$ mod $p$, for all roots $\alpha$ (letting $\rho$ be the half-sum of the
positive roots). Each $p$-alcove contains a $p$-regular (integral) weight if and only if $p\geq h$.  The
set $X_+$ is a poset, setting $\lambda\leq \mu$ if and only if $\mu-\lambda$ is a sum of positive roots. (This is the usual
``dominance order.") Finite 
subposets of $X_+$ play an important role. For example,  $\Resreg$ is the poset ideal in the poset of all $p$-regular dominant weights generated by the set of $p$-restricted $p$-regular weights.

\item[(3)] The Jantzen region is defined by
$$\begin{aligned} \Jan:=\{\lambda& \in X_+\,|\,
(\lambda+\rho,\alpha_0^\vee)\leq p(p-h+2)\\ &{\text{\rm whenever $\alpha_0$ is a 
maximal short root for a component of $\Phi$}}. \}\end{aligned}$$

\item[(4)] The Lusztig character formula (LCF) is a conjectured formula
giving the characters of irreducible modules $L(\gamma)$ in terms of characters of standard (Weyl) modules $\Delta(\lambda)$,  for $\gamma, \lambda, \in \Jan$.  We need not describe this expression explicitly here; see \cite[Appendix C]{JanB}. We note that the LCF holds if  $p$ is sufficiently large, depending on the root system $\Phi$. (See \cite{Fiebig} for a specific bound.)
Also, (\ref{LCF}) below provides an equivalent formulation---convenient for this paper---for the LCF (or, more accurately, for its validity for $\Jan$), stated in terms of all dominant weights.   \end{itemize}
 
\medskip\medskip
\begin{center}{\bf{\Large Part I: A survey}}\end{center}

\vskip.3in
\section{First forced grading results}

 We assume the reader has some familiarity with the theory of rational representations of a reductive group $G$ over an algebraically closed field $k$ of characteristic $p>0$. Nevertheless, the results we discuss here are already meaningful and nontrivial even for $G=GL_n$, a general linear group, or the corresponding ``simple" (and simply connected) algebraic group $SL_n$, the special linear group of all $n\times n$ determinant 1 matrices.

The finite
dimensional
modules with composition factors having highest weights in some fixed finite saturated
set (poset ideal) $\Gamma$ of dominant weights identify with the finite
dimensional modules for a finite
dimensional algebra $A$. The algebra $A$ may be constructed as the quotient $\Dist(G)_\Gamma$ of the distribution algebra $\Dist(G)$ of $G$ by the annihilator in
$\Dist(G)$ of all finite dimensional modules with composition factors having highest weight in $\Gamma$. (It is not entirely obvious that such a quotient is finite dimensional, but it is.) The homological algebra of $A$ closely parallels
that of $G$; for example, given $M,N\in \Amod$, $\Ext^\bullet_A(M,N)\cong
\Ext^\bullet_G(M,N)$, where, on the right-hand side, $M,N$ are identified with rational $G$-modules.  By
varying $\Gamma$, the representation theory of $G$ can be largely recaptured from that of the algebras $A$. Sometimes it is useful to
study smaller categories in this way, such as blocks or unions of blocks of $G$-modules, in which case corresponding statements hold for evident modifications of the posets $\Gamma$ and algebras $A=\Dist(G)_\Gamma$.

All these algebras $A$ are examples of quasi-hereditary
algebras, which have been extensively studied in their own right, but often with an eye toward
applications to representation theory.  When  $G=GL_n$, the famous Schur algebras arise
 this way \cite{GES}. The representation/cohomology
theory of a (Lusztig) quantum enveloping algebra $U_{\zeta}$ at a root of unity also can be studied by
means of similarly defined finite dimensional algebras, using $U_{\zeta}$ in the constructions above. The $q$-Schur algebras of
Dipper-James \cite{DJ} (see also  \cite{Donkin}, \cite{DDPW}  for exhaustive treatments) are the most
 well-known examples.

 A central theme of this article is the suggestion that many (perhaps all)
 of these algebras $A$ have positively graded companion algebras $\wgr A$ with many nice properties. Our interest in graded structures goes back the mid-nineties, especially  to \cite[Remark (2.3.5)]{CPS5}, which observed that the Lusztig conjecture would hold for all the irreducible $A$-modules provided $A$ itself had a sufficiently nice positive grading. In particular it was required that $A$ be isomorphic to $\gr A$, the graded algebra obtained from $A$ by adding together all quotients of successive terms in its radical filtration. (A formal
 definition is given below.) Unfortunately, the same paper gave an example \cite[3.2]{CPS5} that showed an abstract finite dimensional algebra $A$ could have many standard Lie-theoretic
properties, including known ungraded homological consequences of Koszulity, yet fail to have a Koszul grading,  or even have $A\cong \gr A$.  That depressing fact considerably slowed the search for good positive gradings for well over a decade. Finally, we came to the idea of not requiring $A$ be isomorphic to $\gr A$ but only that it share major properties with the latter, such as quasi-heredity. Anyone who has worked with $\gr A$ knows that this requirement sets a very high bar, likely involving difficult proofs, if possible at all. Nevertheless, with a great deal of effort and use of a highly nontrivial result in \cite{CPS5}\footnote{This result is \cite[Thm.2.2.1]{CPS5}. Its hypothesis is a very strong version of the parity conditions on Ext groups common in Kazhdan-Lusztig theory; this strong version holds by \cite[Thm.3.8]{CPS2} if a Koszul grading is also present. The usual parity conditions involving Ext between standard and irreducible modules have to be extended, in the hypotheses of (a dual version of) this theorem, to involve radical series of standard modules. These hypotheses are sufficiently involved that it was not clear it was even practical to verify them without knowing {\it a priori} a Koszul grading was present. Independently of \cite{PS13}, such a verification was first done in a Lie-theoretic setting (for which a Koszul grading on $A$ itself was---and still is---unknown) in the case of (truncated versions of) Virasoro algebras in characteristic 0 \cite{BNW}.} we were able to obtain a result \cite[Thm.10.6]{PS13} of this type. The following theorem gives a slightly abbreviated version. We assume the reader is familiar with the quasi-hereditary algebra notion, but we will give a formal definition of ``standard Koszul" later. (See particularly Definition \ref{standardQKoszul}  below.)

\begin{thm}\label{charpJussieu} Assume that $G$ is a semisimple, simply connected algebraic group over a field of characteristic $p\geq 2h-2$. Also, assume that $p$ is large enough so that the LCF holds for all $p$-regular weights in $\Jan$.  Let $\Gamma$ be a poset ideal in $\Gamma_{\text{\rm res,reg}}$. {\text{\rm [}}Let $N$ be the global dimension of
$A:=\Dist(G)_\Gamma$ or $N=2$ in case this global dimension is 1. Assume also that $p\geq 2N(h-1)-1$.{\text{\rm ]}}
Then the algebra $\gr A$ is a quasi-hereditary algebra with a standard Koszul
grading. Its standard modules  have the form  $\gr\Delta(\lambda)$.
$\lambda\in\Gamma$. \end{thm}

 The condition on $N$ in the statement of the theorem was removed in \cite[Cor. 3.8]{PS13b}, which also
 implies the ideal $\Gamma_{\text{\rm res, reg}}$ can be replaced by the (po)set $\Gamma_{\text{\rm Jan, reg}}$ of
  all $p$-regular weights in the the Jantzen region $\Jan$. The hypothesis on $p$ in any of these formulations
 is satisfied for $p$ sufficiently large, depending on the root system, as noted in (4) above in the Notation.
  The algebra $\gr A$ is explicitly given by
$$ \gr A :=\bigoplus_{n\geq 0}\rad^nA/\rad^{n+1}A .$$
A similar definition may be made for $A$-modules, such as $\Delta(\gamma)$, resulting in $\gr A$-modules, such as $\gr \Delta(\gamma)$.

The algebras $\gr A$ provide the first examples in this paper of the ``forced grading" process. In general terms,
this just means the building of a graded algebra from a filtration by ideals. In the present context, the process itself is quite familiar. The novelty, however, is that it is possible to prove serious theorems about these algebras, either that $\gr A$ inherits an important property from $A$, or has some other important new
property. Later, the ``forced grading" process itself will involve a more sophisticated construction.

A characteristic 0 quantum version of the above theorem was also proved \cite[Thm. 8.4]{PS13}. We state a version below, and then  describe some recent improvements. Let $U_{\zeta}$ be a (Lusztig) quantum enveloping algebra over $\mathbb Q(\zeta)$ with $\zeta^2$ a primitive $e$th root of unity.  Assume $e$ is odd,\footnote {A slightly different notation was used in \cite{PS13}, effectively requiring that $\zeta$ itself, rather than $\zeta^2$, be an odd root of unity. However, this was unnecessary, as follows, for example, from \cite[Thm. 7.3]{PS6}. The current notation fits with later  usage of $e$, as in \cite{PS6} and our discussion in $\S6$ below
of \cite{SVV14}. As discussed there, the assumptions on $e$ and $\zeta$ can be further relaxed.}  and not divisible by three in case the underlying root system $\Phi$ has a component of type $G_2$.  Assume also $e>h$. The notation $\Delta_{\zeta}(\lambda)$ refers to a standard module in the quantum context.

\begin{thm}\label{KoszulforgeneralizedqSchur}  Let $\Gamma$ be a finite non-empty poset ideal in the poset of $e$-regular dominant weights and let $B=U_{\zeta, \Gamma}$. Then the algebra $\gr B$ is a quasi-hereditary algebra with a standard Koszul
grading. Its standard modules  have the form  $\gr\Delta_{\zeta}(\lambda)$,
$\lambda\in\Gamma$.
 \end{thm}

 It is a deeper fact that, under even weaker assumptions on $e$ (none for type $A$), the use of $e$-regular (rather than all) dominant weights is unnecessary, and that $B$ and $\gr B$ become isomorphic upon base change to $\mathbb C$. This was essentially conjectured by us in a remark following \cite[Conj.IIb]{PS14}. We noted there that such a result for type $A$ had been proved by Shan-Varagnolo-Vasserot in \cite{SVV14}, and we speculated that their methods should give similar results in the other types. This is indeed the case, and we state and prove such a general result as Corollary \ref{generalSVV}, after giving an exposition of its main ingredient, the affine Lie algebra result \cite[Thm.3.6]{SVV14}.

Before moving on, we mention that the restricted enveloping algebra $u\subset\Dist(G)$ and its quantum analogue
$u_\zeta\subset U_\zeta$ (usually referred to as the small quantum group) play a key role in this work (though it is
not always mentioned in this survey). Let $u'_\zeta$ denote the sum of the $e$-regular blocks of $u_\zeta$. When
$e>h$, the Lusztig character formula holds for $U_\zeta$ and it is known from \cite{AJS} that $u'_\zeta$ is a Koszul
algebra. The sum of the $p$-regular blocks for $u$ is denoted $u'$ or $\fa$. When $p> h$ and the LCF holds for
$G$, \cite{AJS} also proves that the algebra $u'$ is Koszul. These graded algebras $u'$ (usually denoted by $\fa$) and $u'_\zeta$ (with $e=p$) play a central
role in the proofs of the theorems above. At the end of \S2, we will discuss another compatible (not necessarily Koszul) grading on $\fa$ which does not require that the LCF holds, only that $p>h$. The grading on $\fa$ under the hypothesis
of $p>h$ arises from a grading at the integral level via base change, as will be discussed.

\section{$p$-filtrations}
 Continuing now with the time-line initiated by the above theorems,  we have much more to say about forced gradings.  Actually, the next result \cite[Cor.5.2]{PS5} doesn't mention forced gradings at all, but they are there---in its proof!
 The theorem answers, in the large $p$ case, a question raised by Jantzen \cite{Jan80} more than 35 years ago. 
 As previously remarked, the hypotheses hold for
 any sufficiently large $p$, depending on the root system.

\begin{thm}\label{cor5.2}  Assume that $p\geq 2h-2$ is an odd prime and that the LCF  holds for $G$. Then for any $\gamma \in X_+$, $\Delta(\gamma)$ has a $\rDelta=\Delta^p$-filtration.
\end{thm} 

The modules  $\Delta^p(\gamma)$ date back to Jantzen's paper. If a dominant weight $\gamma$ is written 
$\gamma=\gamma_0 + p\gamma_1$ with $\gamma_0$ $p$-restricted and $\gamma_1$ dominant, then $\Delta^p(\gamma):=L(\gamma_0)\otimes \Delta(\gamma_1)^{[1]}$, the tensor product of a restricted irreducible module with the Frobenius twist of a standard module. The modules $\rDelta(\gamma)$ are obtained by ``reduction mod $p$", using a lattice generated by a highest weight vector (of weight $\gamma$),  from an irreducible module for the quantum enveloping algebra $U_\zeta$. They were first studied substantively by Lin \cite{Lin98} and later in \cite{CPS09} (with $e=p$), which gave them the name used here. Lin's paper shows $\rDelta(\gamma)\cong \rDelta(\gamma_0)\otimes \Delta(\gamma_1)^{[1]}$. The LCF, as a traditional character
formula for irreducible modules, follows in the quantum (root of unity, dominant weight) case from an analog
for negative level affine Lie algebras, whenever the Kazhdan-Lusztig functor (see \cite{T} and the discussion in
\S6 below) is an equivalence. This is true, even in the case of singular highest weights.
Thus, whenever the LCF is known\footnote{There are modest standing assumptions on the root of unity $\zeta$ used in both the cited papers, but these can surely be weakened with more modern references. In particular, it is likely that Lin's result holds in all the known cases where the Kazhdan-Lusztig functor is an equivalence (as discussed briefly in \S6), though we have not checked it.}
 in the quantum case (e.g, the cases where the Kazhdan-Lusztig functor is an equivalence of categories), the validity of the LCF  in a corresponding algebraic groups case, for a given prime characteristic $p>0$, is equivalent to the assertion
\begin{equation}\label{LCF}\rDelta(\gamma)\cong \Delta^p(\gamma),\end{equation}
 for all dominant weights $\gamma$. This frees the LCF from the Jantzen region, and is a formulation that can  sometimes be used in dominant weight poset ideals, without requiring the above isomorphism to hold for all dominant weights. See \cite[Thm.B.1]{PS5} for a modest application of this kind, a variation on the theorem above in which the LCF is not assumed for the weight $\gamma$ itself. For other illustrations, see
 \cite{PS13d}.

For our purposes, here, however, the main advantage of the above equation is that it shows, for $p\gg 0$, that finding  $\Delta^p$-filtrations is equivalent to finding $\rDelta$-filtrations. Since standard modules in all cases
can be obtained by ``reduction mod $p$" from analogous quantum
modules, this allows Jantzen's problem to be considered
from the quantum group point of view.\footnote{Moreover, even for small primes $\rDelta$-filtrations give
$\Delta^p$-filtrations, since the modules $\rDelta(\gamma)$ all have $\Delta^p$-filtrations. However, there is no complete equivalence of the two filtration types for small primes. In fact,  it is known, through an example of Will Turner (see \cite[Prop.5.1]{PS14}) for $p=2$, that there are standard modules that do not have $\rDelta$-filtrations, though they do have $\Delta^p$-filtrations.}

Actually, thinking entirely in terms of the modules $\rDelta$,  a much stronger analog of the Jantzen question had been conjectured in \cite[Conj.6.11]{CPS09} for these primes $p>h$ and $p$-regular dominant weights $\gamma$, namely, that $\Delta(\gamma)$ has a $\rDelta$-filtration compatible with the filtration induced on it through the quantum radical series. 
The result  \cite[Thm. 6.9]{CPS09} established that the top two sections of this filtration had such a $\rDelta$-filtration (the topmost section just being $\rDelta(\gamma)$ itself) assuming $p$ was also large enough that the LCF held. In essence, it is this conjecture that is proved (in the large $p$ case) in \cite{PS5} to obtain the theorem above. First, the proof reduces to the case where the weight $\gamma$ in Theorem \ref{cor5.2} is $p$-regular. Then the standard module $\Delta(\gamma)$ is lifted to an integral form for the quantum group, and a ``radically new" version of forced grading is used.

An explanation of this new version requires the notion, first given in \cite{CPS1a}, of an {\it integral} quasi-hereditary algebra $\wA$ over a commutative ring $\sO$---in our case,  $\sO$ is a DVR. Denote its fraction field by $K$ and its residue field by $k$.  For simplicity, we also assume that  $\wA$ is {\it split} in the sense that all irreducible modules of $\wA_k$ or $\wA_K$ are absolutely irreducible.  Then $\wA$ is a
(split) integral quasi-hereditary algebra if it is, first of all, an $\sO$-algebra which is finitely generated and free over $\sO$(i.e., $\wA$ is an $\sO$-order). Second,  $\wA_k$, $\wA_K$ must be both quasi-hereditary with respect to the same (finite) poset $\Gamma$. Finally, for any $\gamma \in \Gamma$ the standard modules $\Delta_k(\gamma)$ and $\Delta_K(\gamma)$ for both of these algebras are required to be obtained by base change from the same $\wA$-lattice $\wDelta(\gamma)$.\footnote{To obtain agreement with the definition in \cite{CPS1a} it is necessary to build some projective $\wA$ modules with
appropriate $\wDelta$-filtrations. These can be obtained by an iterative construction, adapting the arguments of \cite[2A.3]{DS}. Alternately, \cite[Thm.4.15 ]{Ro} may be quoted.} The new version of forced grading is the construction
$$\gr\wA:=\bigoplus_{i\geq 0}\frac{\wA\cap \rad^i\wA_K}{\wA\cap\rad^{i+1}\wA_K}$$
which gives a positively graded $\sO$-order. A similar construction may be applied to $\wA$-lattices $\wM$, yielding graded lattices $\gr\wM$ for $\gr\wA$. Each of these constructions can be  base-changed to $k$. We sometimes write $A$=$k\otimes\wA$ and $\wgr A$=$k\otimes \gr\wA$, and use similar conventions for lattices. The graded $k$-algebra $\wgr A$ can be quite different from $\gr A$, and, generally speaking, we show it has better properties than the latter. A main result in this direction is \cite[Thm 6.3]{PS5.5}, below, stated at the integral level. Here the algebra $\sO$ is the localization of ${\mathbb Z}[\zeta]$ at the ideal generated by $\zeta -1$, with $\zeta$ a primitive $p$th root of 1. The algebra $\wA=\wU_{\zeta,\Gamma}$ is the natural $\sO$-order
in the truncation $U_{\zeta,\Gamma}$ of the Lusztig quantum group over $\mathbb Q(\zeta)$.

\begin{thm}\label{mainquantumresult} Assume that $e=p>2h-2$ is an odd prime and consider the algebra $\wA=\wU_{\zeta,\Gamma}$, where $\Gamma$
is an non-empty (finite) ideal of $p$-regular dominant weights. Then $\gr\wA$ is a quasi-hereditary algebra over $\sO$, with
standard modules $\gr\,\wDelta(\lambda)$, $\lambda\in\Gamma$.
\end{thm}

The condition on $p$ can likely be improved, and $p$-regularity may not be necessary. Nevertheless, the theorem as stated played a key role
in the proof of Theorem \ref{cor5.2}. Very briefly, an Ext$^1$-vanishing criterion for grade-by-grade  $\rDelta$-filtrations of graded lattices $\gr \wM$ is obtained in \cite[Thm.4.4]{PS5}.  The criterion is that $\grExt^1_{\gr\wA}(\gr\wM,\wX)=0$
for certain graded $\gr\wA$-lattices $\wX$. (Here $\grExt$ denotes Ext computed in the category of graded
modules, a notation that will reemerge again in \S4.) The graded modules $\wX$ used in the second variable are graded $\gr\wA$-lattices arising from injective $\wA_K$-modules.\footnote{The proof of the criterion requires control on the grade 0 part of the base change of these lattice to $K$, most easily obtained
if the base change gives an indecomposable module at the $\wA_K$ level. This requires the LCF.} The construction of
these modules $\wX$ gives them, in the language of the graded integral quasi-hereditary algebra $\gr \wA$, a costandard filtration, so that the required Ext$^1$-vanishing is satisfied when $\gr\wM=\gr\wDelta(\gamma)$. 

To return to the end of \S1, the small quantum group $u_\zeta$ admits an integral form, as does $u'_\zeta$. The
latter is denoted $\wfa$. We have $\wfa_k\cong \fa=u'$. It is proved in \cite[Thm. 8.1]{PS5.5} that, when $p>h$, the
algebra $\wfa$ has a positive grading which base changes to the Koszul grading on $u'_\zeta$. In turn, the grading
on $\wfa$ induces a positive grading on $\fa$.  Also, when the LCF holds, the grading on $\fa$ agrees with its Koszul
grading (cf. the remarks at the end of \S1).

\section{Applications to good filtrations}
The forced grading methods described in the previous section form a basis for extending their scope from  module
structure theory to the study of homological resolutions. In particular, this leads to some striking results concerning good
filtrations of cohomology spaces. The algebra $\fa$ (which captures part of the representation theory of the infinitesimal
group $G_1$) is used as both a tool and a target. For example, the following result is proved in \cite[\S\S5,6]{PS13b}.

\begin{thm}\label{goodfiltrationtheorem} Let $G$ be a semsimple, simply connected algebraic group, and assume that $p\geq 2h-2$ is odd and sufficiently large that the LCF holds.   Let $\lambda,\mu$ be arbitrary $p$-regular
dominant weights.  Then the following statements hold:

(a) For any integer $n\geq 0$,  the rational $G$-module $\Ext^n_{G_1}(\rDelta(\lambda),\rnabla(\mu))^{[-1]}$  has a $\nabla$-filtration.

 (b) For any integer $n\geq 0$, the rational $G$-module $\Ext^n_{G_1}(\Delta(\lambda),\rnabla(\mu))^{[-1]}
$ has a $\nabla$-filtration. Also,  the natural map
$$
\Ext^n_{G_1}(\rDelta(\nu),\rnabla(\mu))\to\Ext^n_{G_1}(\Delta(\nu),\rnabla(\mu))$$
induced by the quotient map $\Delta(\nu)\twoheadrightarrow\rDelta(\nu)$  is surjective.

(c)  Dually, for any integer $n\geq 0$, the rational $G$-module $\Ext^n_{G_1}(\rDelta(\lambda),\nabla(\mu))^{[-1]}$ has a $\nabla$-filtration. Also,
the natural map.
$$ \Ext^n_{G_1}(\rDelta(\mu),\rnabla(\nu))\to\Ext^n_{G_1}(\rDelta(\mu),\nabla(\nu))$$
 induced by the inclusion $\rnabla(\nu)\hookrightarrow\nabla(\nu)$ is surjective.
\end{thm}

In addition, \cite[\S7]{PS13b} gives explicit formulas for the multiplicities
$$
\begin{cases}
[\Ext^n_{G_1}(\rDelta(\lambda), \rnabla(\mu))^{[-1]}:\nabla(\omega)],\\ [\Ext^n_{G_1}(\Delta(\lambda), \rnabla(\mu))^{[-1]}:\nabla(\omega)],\\ [\Ext^n_{G_1}(\rDelta(\lambda),\nabla(\mu))^{[-1]}:  \nabla(\omega)]
\end{cases}
$$
for any dominant weight $\omega$. These multiplicities are expressed in terms of Kazhdan-Lusztig polynomial coefficients for the affine Weyl group
$W_p$ attached to $G$.

\begin{rems} (a) Theorem \ref{goodfiltrationtheorem} is suggested by the work \cite{KLT} of Kumar, Lauritzen, and Thomsen, showing that, if $p>h$, then, for any $n\geq 0$, $H^n(G_1,\nabla(\lambda))^{[-1]}=\Ext^n_{G_1}(k,\nabla(\tau))^{[-1]}$ always has a
$\nabla$-filtration.

(b) When the conclusion of Theorem \ref{goodfiltrationtheorem}(a) holds, observe that
$$\Ext^n_G(\rDelta(\lambda),\rnabla(\mu))\cong (\Ext^n_{G_1}(\rDelta(\lambda),\rnabla(\mu))^{[-1]})^G$$
since the corresponding Hochschild-Serre spectral sequence
$$E_2^{s,t}=H^s(G,\Ext^t_{G_1}(\rDelta(\lambda),\rnabla(\mu))^{[-1]})\implies \Ext^{s+t}_G(\rDelta(\lambda),\rnabla(\mu))$$
collapses in the sense that $E_2^{s,t}=0$ for $s>0$ and $E_2^{0,s+t}\cong E^{0,s+t}_\infty\cong\Ext^{s+t}_G(\rDelta(\lambda),
\rnabla(\mu))=0$. This is true because $H^n(G,\nabla(\tau))$ vanishes for $n>0$ and all dominant weights $\tau$.
In particular, $\Ext^n_{G}(\rDelta(\lambda),\rnabla(\mu))\cong
(\Ext^n_{G_1}(\rDelta(\lambda),\rnabla(\mu))^{[-1]})^G$, which can be determined from the multiplicity results mentioned
above. Similar comments apply to parts (b), (c) of the theorem.  See \cite[Conj.]{APS} for a general conjecture
related to the above observations.

(c) When $p>h$, the dimensions of the spaces $\Ext^n_G(\Delta(\lambda)^{[1]},\Delta(\mu))$, $\Ext^n_G(\Delta(\lambda),
\nabla(\mu)^{[1]})$, and $\Ext^n_G(\Delta(\lambda)^{[1]},\nabla(\mu))$ are determined in \cite[Thm. 5.4]{CPS09}. In this
case, $\rDelta(p\lambda)\cong\Delta(\lambda)^{[1]}$ and $\rnabla(p\lambda)\cong\nabla(p)^{[1]}$.
\end{rems}

An essential ingredient in the proof of Theorem \ref{goodfiltrationtheorem} involves the construction of certain
explicit exact complexes $\Xi_\bullet\twoheadrightarrow M$. Here $M$ is a graded $\wgr A$-module, where
$A=A_\Gamma$ for a finite ideal $\Gamma$ of $p$-regular dominant weights. There is an increasing sequence
$\Gamma\subseteq\Gamma_0\subseteq\Gamma_1\subseteq\cdots$ of finite ideals of $p$-regular weights such that
each term $\Xi_i$, $i\geq 0$,
is a graded $\wgr A_{\Gamma_i}$-module, and the map $\Xi_i\to\Xi_{i-1}$ is a morphism in $\wgr A_{\Gamma_i}$-grmod.
This makes sense because $\Xi_{i-1}$ can be regarded as a graded $\wgr A_{\Gamma_i}$-module through the
surjective algebra homomorphism $\wgr A_{\Gamma_i}\twoheadrightarrow\wgr A_{\Gamma_{i-1}}$.
In addition, when
$M$ is regarded as a graded $\wfa$-module, it is required that it be $\fa$-linear.\footnote{The term ``linear" generally refers to a certain kind of graded complex , though we often use it to refer to graded objects which have a resolution by such a complex, the nature of which depends on context. Typically, the objects in the $i^th$ term of any linear complex are shifts $X_i=X\langle i\rangle$ of graded objects $X$ which have some preferred or standard form for their grading--e.g., a projective object generated in grade 0, or an injective object with socle of grade 0.  Mazorchuk [Mazor10] even defines linear complexes of tilting modules. One useful aspect of having a terminology "linear" for objects, and not just complexes, is that it is often useful to consider resolutions of linear objects which may not be linear, as we are doing here.}
 In other words, if $P_\bullet\twoheadrightarrow M$ is a minimal
graded $\fa$-projective resolution, then $\ker(P_{i+1}\to P_i)$ is generated by its terms of grade $i+2$. When $p$
is sufficiently large that the LCF holds, then linearity holds for the modules $\rDelta(\lambda)$ (since
$\fa$ is Koszul). In addition, the standard modules $\Delta(\lambda)$ for $p$-regular dominant weights $\lambda$ are also
linear (see Theorem \ref{theorem6.3} below).  Further, the resolution can be chosen so that
each $\Xi_i$, when viewed as an $\fa$-module through the map $\fa\to\wgr A_{\Gamma_i}$, is projective. Also,
the $\Xi_i$ and syzygy modules $\Omega_{i+1}=\ker(\Xi_i\to\Xi_{i-1})$ have $\rDelta$-filtrations.

We conclude this section with two final results which will be important later. Observe that part (a) of the
first theorem and all the second theorem do not require any assumptions about the LCF. The proof of the first
theorem is given in \cite[Thm. 6.3]{PS13b} and the proof of the second is found in \cite[Thm. 5.3(b), Thm. 6.5]{PS13b}.

\begin{thm}\label{theorem6.3}  Assume that $p\geq 2h-2$ is odd.

(a) For a $p$-regular dominant weight $\lambda$, the standard module $\Delta(\lambda)$ has a graded $\fa$-module
structure,
 isomorphic to $\wgr\Delta(\lambda)$
over $\wgr\fa\cong\fa$.

(b) Assume that the Lusztig character formula holds. With the graded structure given in (a), $\Delta(\lambda)$ is linear over the graded algebra $\fa$.
\end{thm}

\begin{thm}\label{comparisonofcohomology} Assume that $p\geq 2h-2$ is odd.  Let $\lambda,\mu$ be $p$-regular
dominant weights contained in a finite ideal $\Gamma$ of $p$-regular weights. Then there are
graded isomorphisms

$$\begin{cases}\Ext^\bullet_{G}(\Delta(\lambda),\rnabla(\mu))\cong\Ext^\bullet_A(\Delta(\lambda),\rnabla(\mu))\cong
\Ext^\bullet_{\wgr A}(\wgr \Delta(\lambda),\rnabla(\mu)),\\
\Ext^\bullet_G(\rDelta(\lambda),\rnabla(\mu))\cong\Ext^\bullet_A(\rDelta(\lambda),\rnabla(\mu))
\cong\Ext^\bullet_{\wgr A}(\rDelta(\lambda),\rnabla(\mu)),\\
\Ext^\bullet_G(\rDelta(\lambda),\nabla(\mu))\cong\Ext^\bullet_A(\rDelta(\lambda),\nabla(\mu))
\cong\Ext^\bullet_{\wgr A}(\rDelta(\lambda),\wgr^\diamond\nabla(\mu)),\end{cases}
$$
  where $\wgr^\diamond\nabla(\mu)$ denotes the
dual Weyl module for $\wgr A$ of highest weight $\mu$. \end{thm}

In the process of proving these results, new graded homological properties of reduced standard and costandard modules emerged. These
are of a graded nature, and their
description is best described in the context of a new class of graded algebras, which we now introduce.

\section{Q-Koszul algebras} This section introduces the notion of a Q-Koszul algebra and a standard Q-Koszul
algebra \cite{PS13b}, \cite{PS13d}, \cite{PS14}. We work with
finite dimensional algebras $A$ over a field $k$ of characteristic $p\geq 0$. Generally, we assume that $A$ is split
over $k$ in the sense that the irreducible $A$-modules remain irreducible upon extension to $\overline k$. Also, assume that $A$
has an $\mathbb N$-grading
$$ A=\bigoplus_{i=0}^\infty A_j,\quad A_iA_j\subseteq A_{i+j}.$$
The subspace $A_i$ is called the term in grade $i$. The nilpotent ideal
$A_{>0}:=\bigoplus_{i>0} A_i$ plays an important role. The quotient algebra $A/A_{>0}$ is isomorphic to the subalgebra
$A_0$ of $A$. In general, $A_0$ is {\it not} assumed to be semisimple.

Let $\grAmod$ be the category of finite dimensional $\mathbb Z$-graded
$A$-modules $M$. Thus,
$$M=\bigoplus_{i\in\mathbb Z}M_i, \quad A_jM_i\subseteq M_{i+j}.$$
Also, $\Amod$ is the category of finite dimensional $A$-modules (no grading). There is a natural forgetful functor
$\grAmod\to\Amod$. Let $\Lambda=\Lambda_A$ be a fixed finite set indexing the distinct isomorphism classes of irreducible $A$-modules:
$$\lambda\in\Lambda \longleftrightarrow L(\lambda)\in\Amod.$$
If $M\in\grAmod$ and $r\in\mathbb Z$, then $M\langle r\rangle$ is the graded $A$-module obtained from $M$
by ``shifting" the grading $r$ steps to the right: $M\langle r\rangle_i:=M_{i-r}$. Clearly, the set $\Lambda\times {\mathbb Z}$
indexes the isomorphism class of irreducible objects in $\grAmod$: if $r\in\mathbb Z$ and $\lambda\in\Lambda$,
the pair $(\lambda,r)$ corresponds to the graded irreducible module $L(\lambda)\langle r\rangle$. Up to isomorphism
every graded irreducible module is isomorphic to a unique $L(\lambda)\langle i\rangle$. The set $\Lambda$ also
indexes the isomorphism classes of irreducible $A_0$-modules.

The categories $\Amod$ and $\grAmod$ both have enough projective and injective modules. If $P(\lambda)$ denotes
the PIM of $L(\lambda)$ in $\Amod$, then $P(\lambda)$ has a unique  $\mathbb N$-grading making it an object in
$\grAmod$ such that $P(\lambda)_0\cong L(\lambda)$ as a $A_0$- or a $A/A_{>0}$-module. Given $M,N\in\grAmod$,
the Ext-groups in $\grAmod$, are denoted $\grExt^n_A(M,N)$, $n=0,1,\cdots$. These are related to the $\Ext$-groups
in $\Amod$ by means of the identity:
$$\Ext^n_A(M,N)\cong \bigoplus_{r\in\mathbb Z}\grExt^n(M,N\langle r\rangle).$$

\begin{defn}\label{basicdefn} Let $A$ be an $\mathbb N$-graded algebra, but assume that there is a poset structure on $\Lambda$
with respect to which $A_0$ is a quasi-hereditary algebra. For $\lambda\in\Lambda$, let $\Delta_0(\lambda)$
(resp., $\nabla^0(\lambda)$) denote the corresponding standard (resp., costandard) $A_0$-module, regarded as
a graded $A$-module. Then $A$ is called Q-Koszul provided, for all integers $i>0$,
\begin{equation}\label{QKproperty}\grExt^i_A(\Delta_0(\lambda),\nabla^0(\mu)\langle j\rangle)\not=0\implies i=j, \quad\forall j\in{\mathbb Z},
 \lambda,\mu\in\Lambda.\end{equation}
 If $n\geq 0$ and (\ref{QKproperty}) holds for $0<i\leq n$, then $A$ is called $n$-Q-Koszul.
 \end{defn}

 Thus, $A$ is Q-Koszul if it is $n$-Q-Koszul for all $n>0$. The following result, found in \cite[Thm. 2.3]{PS14},
 shows that the properties of being
1- or 2-Koszul are already quite strong.

\begin{thm} (a) Assume that $A$ is 1-Q-Koszul as explained in Definition \ref{basicdefn}. Then $A$ is tight, in the
sense that for any $n\geq 1$, $A_n$ is the product $\underbrace{A_1\cdots A_1}_n$.

\smallskip
(b) Assume that $A$ is 2-Q-Koszul. Then $A$ is a quadratic algebra, in the following sense. Let
$$T_{A_0}(A_1)=\bigoplus_{n=0}^\infty T_{A_0}^n(A_1),\quad{\text{\rm where}}\,\, T_{A_0}^n(A_1):=A_1\otimes_{A_0}\otimes_{A_0}\cdots\otimes_{A_0}A_1$$
be the tensor algebra of the $(A_0,A_0)$-bimodule $A_1$. Then the mapping $T_{A_0}(A_1)\to A$,
$a_1\otimes\cdots\otimes a_n\mapsto a_1\cdots a_n$, is surjective with kernel generated by its terms in grade 2.
\end{thm}

Now let $A=\bigoplus_{n\geq 0}A_n$ be a (positively) graded quasi-hereditary algebra with weight poset $\lambda$. It is
elementary to show that $A_0$ is also quasi-hereditary with weight poset $\lambda$. The standard (resp., costandard)
modules for $A_0$ are just the grade 0 components $\Delta^0(\lambda)$ (resp., $\nabla_0(\lambda)$) of
the standard (resp., costandard) modules of $A$.

\begin{defn}\label{standardQKoszul} The  positively graded algebra $A$ is a {\it standard Q-Koszul algebra} provided
that, for all $\lambda,\mu\in\Lambda$,
\begin{equation} \label{conditions}
 \begin{cases}(a)\quad\grExt^n_A(\Delta(\lambda),\nabla_0(\mu)\langle r\rangle)\not=0
\implies n=r;\\
(b)\quad \grExt^n_A(\Delta^0(\mu),\nabla(\lambda)\langle r\rangle)\not=0\implies n=r\end{cases}\end{equation}
for all integers $n,r$.\end{defn}

The grading on $A$ matters: Observe that any quasi-hereditary algebra $A$, given the trivial grading $A=A_0$ is Q-Koszul and even standard Q-Koszul. A Koszul algebra
is Q-Koszul, a more nontrivial example. Following Mazorchuk \cite{Mazor1}, a Koszul algebra $A$ is called standard Koszul provided that $A$
is quasi-hereditary, and if each standard (costandard) module $\Delta(\lambda)$ (resp., $\nabla(\lambda)$) is linear.
Thus, if $\Delta(\lambda)$ is given the unique grading with its head in grade 0, then it has a (graded) projective resolution
$P_\bullet\twoheadrightarrow \Delta(\lambda)$ in which the head of $P_i$ has grade $i$. A dual property is required for
the costandard modules. In particular, a standard Koszul algebra is standard Q-Koszul.

The following theorem, proved in \cite[Cor. 3.4]{PS14}, improves on the definition in \cite{PS13b}, showing that
standard Q-Koszul algebras as defined above are automatically Q-Koszul.

\begin{thm}\label{betterdefn} Let $A$ be a standard Q-Koszul algebra with weight poset $\lambda$. For
$\lambda,\mu\in\Lambda$,
$$\grExt^n_A(\Delta^0(\lambda),\nabla_0(\mu)\langle r\rangle)\not=0\implies n=r$$
for all $n\geq 0$ and all $r\in\mathbb Z$. In particular, $A$ is Q-Koszul.  \end{thm} 

The results and methods described in \S3 now come together to prove the following important theorem.

\begin{thm}\label{bigtheorem} (\cite[Thm. 3.7]{PS13b}) Assume that $p\geq 2h-2$ is odd and that the LCF holds for the semisimple, simply
connected algebraic group
$G$. Let $\Lambda$ be a finite ideal of $p$-regular dominant weights, and put $A:=A_\Lambda$.  The graded algebra $\wgr A$ is standard Q-Koszul with poset $\Lambda$. In addition, the standard modules for $\wgr A$
are the modules $\wgr\Delta(\lambda)$, $\lambda\in\Lambda$. Also, the standard modules for $(\wgr A)_0$ are the
modules $\rDelta(\lambda)$, $\lambda\in\Lambda$. \end{thm}

The costandard modules for $\wgr A$ arise as certain linearly dual modules of right standard modules. The costandard modules
for  $(\wgr A)_0$ are just the $\rnabla(\lambda)$, $\lambda\in\Lambda$.

To prove Theorem \ref{bigtheorem}, it is necessary to prove (\ref{conditions}). As discussed earlier, the algebra
$\wgr A$ is known to be quasi-hereditary with weight poset $\Lambda$ and standard objects $\wgr\Delta(\lambda)$,
$\lambda\in\Lambda$. By \cite{PS5}, each section of $\wgr \Delta(\lambda)$ has a $\rDelta$-filtration. In addition,
Theorem \ref{theorem6.3} implies that $\wgr\Delta(\lambda)$ is linear as an $\fa$-module.  The required vanishing
in (\ref{conditions}) can then be obtained by using properties of the complex $\Xi\twoheadrightarrow\wgr \Delta(\lambda)$
briefly described in the previous section. A similar argument works for the costandard modules.

In \cite{PS14} the authors made several conjectures. A prime $p$ is called KL-good (for a given root system $X$) provided the Kazhdan-Lusztig functors $F_{\ell}$ associated to $X$ (see \cite{T} and \cite{KL} for the definition; a few details
are given below above Corollary \ref{generalSVV})  are category equivalences for
\begin{equation}\label{L}\ell = \ell(p):=\begin{cases} 4\quad p=2;\\ p\quad p\not= 2.\end{cases}\end{equation}
 So, in all cases, if $\zeta$ is a primitive $\ell(p)$th root of unity, $\zeta^2$ is a primitive $p$th root of unity.
 As discussed further in \S6,
 a list of known KL-good primes is given in \cite{T} for each indecomposable type. (More precisely, values of $\ell$ are described for which the Kazhdan-Lusztig functor $F_\ell$ is known to be an equivalence.)  If the root system $X$ is not indecomposable, call a prime KL-good  for $X$ provided it is KL-good for each
component.

The following conjecture is Conjecture I in \cite{PS14}. Evidence for  this conjecture is given, for very large primes,
by Theorem \ref{bigtheorem} above and, for $p=2$, in the theorem below the conjecture.

\begin{conj}\label{conjI} Let $G$ be a semisimple, simply connected algebraic group defined and split over ${\mathbb F}_p$ for a KL-good prime $p$. Let $\Gamma$ be a finite ideal of dominant weights and form the quasi-hereditary algebra  $A=A_\Gamma$.
Then $\wgr A$ is standard Q-Koszul. \end{conj}

\begin{thm}\cite[Thm. 6.2]{PS14}  Let $S(5,5)$ be the Schur algebra for $GL_5(\overline{\mathbb F}_2)$ in
characteristic 2. Let $A$ be the principal block of $S(5,5)$. Then $\wgr A$ is standard Q-Koszul. \end{thm}

It is also stated in \cite{PS14}, without additional details, that the same result holds for the full algebra $S(5,5)$.

\vskip.4in
\begin{center}{\bf\Large Part II: New results}\end{center}

\vskip.2in
\section{\bf Morita equivalence and positive gradings}

Finite dimensional algebras $A$ behave quite well with respect to the interaction of Morita equivalences and positive gradings. For instance, \cite[Lem.F.3]{AJS}, stated for  Artinian rings, shows that a finite dimensional algebra $A$ has a Koszul grading if it is Morita equivalent to an algebra $B$ with a Koszul grading.\footnote{Note that, if $A$ is a finite dimensional algebra over a field $k$, and $B$ is any ring Morita equivalent to $A$, then $B$ is also a finite dimensional algebra over $k$.} The algebra  $A$ is not assumed a priori to have any grading at all. It is the aim of this section to show that many more properties with a positive grading underpinning carry over under Morita equivalence. An especially fundamental one is the existence itself of a comparable positive grading. The proposition below makes this property precise and provides a framework for handling many others.

Recall that $\Amod$  denotes the category of finite dimensional left $A$-modules, and, if $A$ is graded, $\Agrmod$ denotes the category of finite dimensional graded $A$-modules. If both $A$ and $B$ are finite dimensional graded algebras over the same field, and $F: \Bgrmod \to \Agrmod$ is an (additive) functor, we say (following \cite[app. E]{AJS}) that $F$  is {\it graded} if it commutes (up to natural isomorphism) with the grading shift functors, i.~e., $F(M\langle r\rangle)
\cong F(M)\langle r\rangle$, for $M\in\Bgrmod$, $r\in\mathbb Z$.
 We'll also say a graded functor $F$ is a {\it graded version} of an (additive) functor $E:\Bmod\to \Amod$ if there is a functor composition diagram
\begin{equation}\label{diagram}\begin{CD}  \Bgrmod @>{F}>> \Agrmod \\
@VvVV @VvVV\\
  \Bmod @>{E}>>\Amod \end{CD}\end{equation}
  commutative up to a natural isomorphism, in which the vertical maps are forgetful functors (both denoted $v$, by abuse of notation). It is also useful to have the notion of a {\it grade-preserving} functor. This is a graded functor $F$ which takes any graded object $M$ whose nonzero grades $M_n$ all satisfy any given inequality $a\leq n\leq b$ to an object $F(M)$ with the same property. In the case of an exact graded
  functor $F$, this just reduces to the condition that $F(M)$ is pure of grade $n$ whenever $M$ is pure of grade $n$.

  In any case, we only use the term grade-preserving for graded functors. If a grade-preserving functor $F$ is a
  graded version of a functor $E$, as above, we will sometimes simply say that $F$ is a grade-preserving version of
  $E$.

  \begin{prop}\label{Moritagraded} Suppose $A$ is a finite dimensional algebra Morita equivalent to an algebra $B$ which has a given positive grading. Then there is positive grading on $A$ and a grade-preserving functor $F: \Bgrmod \to \Agrmod$ which is a graded version of a functor $E: \Bmod \to \Amod$ with both $F$ and $E$ equivalences of categories. Moreover, $F,E$ and inverse equivalences may be chosen so that the inverse of $F$ is a grade-preserving version of the inverse of $E$.
  \end{prop}

  \begin{proof} First, recall the well-known result from Morita theory that $A\cong eM_n(B)e$ for some positive integer $n$ and ``full" idempotent $e\in M_n(B)$. (That is, $M_n(B)e$ is a progenerator for $M_n(B)$.) Here $n$ is some positive integer and $M_n(B)$ is the algebra of $n\times n$ matrices over $B$. We will also write $M_{n,m}(X)$ for the set of all $n\times m$ matrices with entries from a set  $X$,  for any positive integers $n,m$. Generally, $X$ will have some kind of left or right $B$-module structure, leading to a corresponding structure for $M_{n,m}(X)$ over $M_n(B)$ or $M_m(B)$, respectively.

  Next observe that $M_n(B)$ transparently inherits $B$'s given positive grading. There is a standard Morita equivalence from $B$-mod to $M_n(B)$-mod,  given by tensoring over $B$ with
  $M_{n,1}(B)$. It sends a left $B$-module $N$ to the left $M_n(B)$-module $M_{n,1}(N)$. If $N$ is graded, then its graded structure is obviously inherited, so that the same recipe defines a grade-preserving graded analog of the original functor. Both graded and ungraded versions are equivalences of categories. (As is well-known in the ungraded case, an inverse functor is given by tensoring over $M_n(B)$ with $M_{1,n}(B)$. Identifying $B$ with $eM_n(B)e$, where $e$ is the
  matrix unit $e_{1,1}$, this inverse functor is naturally equivalent to mutlipication by $e$.  The same recipe gives an inverse functor at the graded level, a grade-preserving version of the same
  multiplication functor, essentially identical to it.)

  The proposition will now follow by taking compositions of functors, if we can prove it for the case where $A=eBe$ for an idempotent $e$ with $Be$ a progenerator for the category of $B$-modules. Henceforth, we consider that case. Another reduction we can make is to replace $A$ with any isomorphic algebra, and we proceed to construct a very useful one. Note that $B_0$ is isomorphic to the factor algebra $B/B_{\geq 1}$ by projection, so that there is an idempotent $e_0\in B_0$ which has the same
  projection as $e$. Since the ideal $B_{\geq 1}$ is clearly nilpotent, the projective modules $Be_0$ has the same head as
  $Be$, and so is isomorphic to it. Therefore, the endomorphism algebra $\End_B(Be_0)\cong (e_0Be_0)^{\op}$ is isomorphic to the endomorphism algebra $\End_B(Be)\cong (eBe)^{\op}=A^{\op}$. So we may assume $A=e_0Be_0$, or, equivalently, we may assume $e=e_0$.

 In particular, $A$ inherits $B$'s given positive grading. Also, if $N\in\Bgrmod$, then $eN$ inherits a graded $A$-module structure. The resulting functor $F:\Bgrmod\to \Agrmod$, given by $N\mapsto eN$, is obviously a grade-preserving version of its ungraded analog $E$. The  ungraded functor 
 $E$ is well-known to be a Morita equivalence, with inverse given by $E^\dagger:=Be\otimes_{eBe}(-)$.

 We next construct a graded version $F^\dagger$ of this inverse. Given a graded $eBe$-module $Y$, let $Z=Be\otimes_{eBe}Y$. Let $k$ denote the ground field, and regard $Z$ as the quotient of $V:=Be\otimes_kY$; we will henceforth omit the $k$ subscript. The space $V$ becomes a graded $B$-module in an obvious way, if, for any integer $n$, we let $V_n$ denote the sum of all terms $(Be)_i\otimes Y_j$ with $i,j$ integers such that $i+j=n$. Let $R_n$ denote the $k$-span in $V_n$ of all expressions $st\otimes y-s\otimes ty$ with $s\in (Be)_{i'}, t\in (eBe)_m$ and $y\in Y_{j'}$ for some integers $i',m,j'$ with $i'+m+j'=n$. Then $R=\oplus_n R_n$ is a graded $B$-submodule of $V$. Also, $R$ is precisely the $k$-span of all elements $st\otimes y-s\otimes ty$ with $s\in Be, t\in eBe$ and $y\in Y$. It follows that the natural map of $B$-modules $Y\to Z$ has kernel $R$. This gives $Z=Be\otimes_{eBe}Y$ the structure of a graded $B$-module with $Z_n=Y_n/R_n$ for any integer $n$. Finally, if $f: Y\to Y'$ is a map of graded $eBe$-modules, then $1_{Be}\otimes f$ gives a graded map $V\to V'$, where $V':=Be\otimes Y'$. The image of $R$ is contained, grade by grade, in $R'$, defined by analogy with $R$. Thus, there is an induced graded map $V/R\to V'/R'$. As before, we identify $Z'=V'/R'$ with $Be\otimes_{eBe}Y'$. The map $V/R\to V'/R'$ induced by 
 $1_{Be}\otimes f$ now agrees, with these identifications, with the map $1_{Be}\otimes_{eBe} f$. This gives the
 desired graded version
 $F^\dagger$, now defined on both objects and maps, of the inverse functor $E^\dagger=Be\otimes_{eBe}(-)$.

 It remains to check that this graded version $F^\dagger$ is grade preserving. It is exact, since the ``fullness" of $e$ (in a sense analogous to that explained at the start of this proof) can be stated as the symmetric condition $B=BeB$. This implies that multiplication on the right gives an equivalence from mod-$B$ to mod-$eBe$, so that, in particular $Be$ is a projective right $eBe$-module. It suffices, now, to apply the  graded inverse functor to a graded irreducible $eBe$-module $L$ of pure grade 0, and determine that the result is again pure of grade 0. In view of the graded structure defined above on $Be\otimes_{eBe}Y$ in the case $Y=L$, it is enough to show that $B_{\geq 1}(Be\otimes_{eBe}L)=0$. However,  $B_{\geq 1}Be=(BeB)(B_{\geq 1})e \subseteq B(eB_{\geq 1}e)$, and, for many reasons, $eB_{\geq 1}eL=0$. (One can use grade considerations, or the very general fact that $e\rad(Be)\subseteq \rad(eBe)$.) Thus, the graded version $F^\dagger$ that we have constructed of the inverse functor $E^\dagger$ is also grade-preserving. This completes the proof of the proposition. \end{proof}

 It is also known from \cite{CPS1a} that, if a finite dimensional algebra $B$ is quasi-hereditary and has a positive grading, then
 irreducible, projective and standard modules can be chosen to have a graded $B$-module structure, with heads of grade 0. Left-right symmetry of the quasi-hereditary property implies costandard modules and injective modules also have gradings, with socles of grade 0.
 It is pointed out in \cite{SVV14} that such gradings, with the grade 0 properties given, are unique, and similar unique gradings are noted for indecomposable tilting modules, cf. \cite[Prop.2.7]{SVV14}.\footnote{One way to phrase the uniqueness condition in the case of an indecomposable tilting module is to say its irreducible section of highest weight can be found in grade 0.} As a consequence, if $B$ is quasi-hereditary and positively graded, and $A$ is Morita equivalent of $B$, then all these graded modules get carried by the functor $F$ above into corresponding graded modules (irreducible, standard, projective, constandard, injective, tilting) for $A$, with the characteristic grade 0 property carrying over when present.

  \begin{cor}\label{MoritaCarryOver}  Suppose $A$, $B$ are finite dimensional Morita equivalent algebras, and that $B$ has one or more of the properties below (some of which require a positive grading on $B$). Then $A$ has the corresponding property or properties:\footnote{The Mazorchuk property of a ``balanced" quasi-hereditary algebra $B$ is heavily used in \cite{SVV14}. It means that $B$ is a positively graded quasi-hereditary algebra. Additionally, it is assumed that all
  standard modules  $\Delta(\lambda)$ have linear tilting resolutions $\Delta(\lambda)\hookrightarrow T^\bullet$ and all costandard modules $\nabla(\lambda)$ have
  linear tilting coresolutions $T'_\bullet\twoheadrightarrow \nabla(\lambda)$. See Mazorchuk \cite[p. 3]{Mazor} for a
  precise definition. Note that the shift functors $\langle r\rangle$ in \cite{Mazor} are what we denote by $\langle -r\rangle$.
  (Our notation agrees with that in \cite{BGS}.
  Consequently, Proposition \ref{Moritagraded} shows  ``balanced" is also preserved by Morita equivalence.)}

  \medskip\medskip
       quasi-hereditary, positively graded, Koszul, standard Koszul, Q-Koszul, standard Q-Koszul.

     \medskip\medskip\noindent
      Moreover, after giving $A$ an appropriate positive grading, the categories $A$-grmod and $B$-grmod are equivalent
     by a graded functor.
 \end{cor}

 \begin{proof} Assume the hypothesis on $A$, $B$ and apply Proposition \ref{Moritagraded}. Most of the corollary is immediate from the proposition and remarks after it. We treat the slightly more involved Q-Koszul cases: First, just assuming $B$ is positively graded, note the category of $B_0$-modules identifies with the category
 of graded $B$-modules which are pure of grade 0. It follows that $A_0$ is Morita equivalent to $B_0$, with $F$ furnishing the required category equivalence. In particular, if $B_0$ is quasi-hereditary, then so is $A_0$, with standard and costandard modules, viewed as pure grade 0 modules for $B$ and $A$, respectively, corresponding under $F$. If also $B$ is Q-Koszul, these modules satisfy the defining ext properties (\ref{QKproperty}) in $B$-grmod, which carry over to precisely to the same properties in $A$-grmod. Thus, $A$ is also Q-Koszul.  A similar argument shows that the standard Q-Koszul property also carries over.

 The other cases are much easier. The rest of the proof is left to the reader.
 \end{proof}

\section{\bf Some complements to \cite{SVV14}}

We begin this section with a detailed exposition of a main theorem \cite[Thm. 3.12]{SVV14} of Shan-Varagnolo-Vasserot. In it these authors study, for affine Lie algebras $\bold g$,  certain categories
\begin{equation}\label{display}
^w\mathbb O^{\nu}_{\mu,-} \,\,\,\text{and}\,\,\, ^v\mathbb O^{\mu}_{\nu,+}\end{equation}
 associated to, and collectively determining, blocks at non-critical levels of associated integral weight parabolic categories $\mathbb O$. They show these categories  are equivalent to finite dimensional module categories of finite dimensional $\mathbb C$-algebras which are standard Koszul, even ``balanced" (see the ftn. 10 in \S5 above). We will explain below the notation, but should say right away that, with an appropriate relation between the parameters $w$ and $v$ (elements of the ambient
 affine Weyl group), the two categories are Koszul dual to each other. (This means, if the respective ungraded module categories are equivalent to $\Amod$ and $B$-mod respectively, then $A$ is Morita equivalent to $B^!=\Ext^\bullet_B(B_0,B_0)^{\text{\rm op}}$. It seems harmless, in view of \S5, to say the category $\Amod$ is Koszul or standard Koszul if the algebra $A$ has the corresponding property.) On both sides of (\ref{display}) the symbols $\mu$ and $\nu$ denote proper subsets, possibly empty, of fundamental roots. (Neither symbol $\mu$ or $\nu$ has precedence over the other, and our usage tends to be the reverse of that in \cite{SVV14}.) These are used to parameterize the ``parabolic" and ``singular" features of the ambient parabolic blocks, denoted $\mathbb O^{\nu}_{\mu,-}$ and $\mathbb O^{\mu}_{\nu,+}$, respectively. The subset symbols $\mu$ and $\nu$ used for a superscript parameterize the ``parabolic" nature of the block, and parameterize its ``singular" aspect when used as a subscript. The signs $-$ and $+$ stand for two cases for the level of the block, whether it is $-e-g$ or $+e-g$ where $e$ is a positive integer and $g$ denotes the dual Coxeter number (called $N$ in \cite{SVV14}). The critical level $-g$ is not allowed. The level itself is suppressed in the notation, since the equivalence class of the categories considered depends only on whether the level is below or above $-g$ (as given by the signs $-$ and $+$, respectively) together with other parameters that do not depend on the level. This is a general feature of the notation for blocks here, which is used in a generic way to describe any one of a family of blocks which are all equivalent as categories.

This is made more concrete in \cite{SVV14} by choosing weights $o_{\mu,-}$ and $o_{\nu,+}$ in the anti-dominant and dominant cones (with shifted origin), respectively, so that the irreducible modules in $\mathbb O^{\nu}_{\mu,-}$ and $\mathbb O^{\mu}_{\nu,+}$, respectively, all have  highest weights in the orbit of $o_{\mu,-}$ or $o_{\nu,+}$, respectively, under the dot action of the affine Weyl group. The notation is chosen so that the respective stabilizers of these two weights are $W_{\mu}$ and $W_{\nu}$, the (finite) Weyl groups associated to the respective root systems generated by $\mu$ and $\nu$. This captures the singularity
of the blocks under consideration.
The weights in the first orbit specific to $\mathbb O^{\nu}_{\mu,-}$ are those which are ``$\nu$-dominant", in the sense that their coefficients are nonnegative at any fundamental weight associated to $\nu$. This describes the ``parabolic" aspect of this block. Similar considerations apply to the second orbit, with the roles of $\mu$ and $\nu$ reversed. It is now fairly clear from translation arguments
that the underlying categories of these two blocks are determined up to equivalence by $\mu,\nu$ and the choice of sign. (The arguments in \cite[Sec.6]{PS6} show how translations available only ``in one direction" can give equivalences using highest weight category theory. Alternately, instead of quoting translation arguments, one can argue from Fiebig's excellent characterizations of blocks at non-critical levels for the full category $\mathbb O$, \cite[Thm.11]{Fiebig06}. See Proposition \ref{propabove} below and the discussion following it.)

Next, we discuss the categories on each side of (\ref{display}). The left-hand side is a full subcategory of its
ambient parabolic block, closed under extension, a highest weight subcategory associated to a finite poset ideal, in the original sense of \cite{CPS-1}. The right-hand side of the display is a quotient category, associated with a finite poset coideal. (It is a highest weight category, though the block itself does not exactly fit this formalism, failing to have enough injectives. However, the arguments in \cite[Thm.3.5(b)]{CPS-1} can be used to construct the quotient, using projectives, rather than injectives.) The paper \cite{SVV14} does not generally require their various categories $\mathbb O$ to consist of finitely generated objects, but this adjustment is needed (and used) here
to make sure all objects in the quotient have finite length. (As noted in \cite{SVV14} the adjustment is not entirely necessary, with its omission just giving a quotient identifying with the category of all modules for a finite dimensional algebra, rather than the category of finite dimensional modules.)

Finally, we discuss the parameters $w$ and $v$ and the poset ideals and coideals they control.   Again following \cite{SVV14}, let $I_{\mu}^{\text{max}}$ denote the set of maximal length left coset representatives (in the affine Weyl group, denoted $\widehat W$ in \cite{SVV14})  for $W_{\mu}$, and define $I_{\nu}^{\text{min}}$ as the set of minimal length left coset representatives for $W_{\nu}$. Define $(I_{\mu,-},\preceq )$ to be the poset $I_{\mu}^{\text{\rm max}}$ with order relation $\preceq$ the Bruhat order, and define $(I_{\nu,+},\preceq)$ to be the poset $I_{\nu}^{\text{min}}$ with order relation $\preceq$ the opposite Bruhat order. Next, define $I_{\mu,-}^{\nu}$ to be the subposet  consisting its elements $y$ for which $y\cdot o_{\mu,-}$ is $\nu$-dominant, and define $I_{\mu,+}^{\nu}$ analogously, interchanging the roles of $\mu$ and $\nu$, and of $+$ and $-$. Finally, for $w\in I_{\mu,-}^{\nu}$ let $^wI_{\mu,-}^{\nu}$ denote the (finite) poset ideal it generates, and, for $v\in I_{\nu,+}^{\mu}$ let $^vI_{\nu,+}^{\mu}$ denote the (finite) poset coideal it generates.  At this point, it is useful to pause, and note that these posets can all be used to index irreducible modules, taking dot products with $o_{\mu,-}$ or $o_{\nu,+}$, as appropriate, to get highest weights for them. We can now describe the left hand side of (\ref{display}) as the full subcategory, closed under extension, of $\mathbb O^{\nu}_{\mu,-}$  generated by the irreducible modules indexed by $^wI_{\mu,-}^{\nu}$. Similarly, the right-hand side is the quotient category of $\mathbb O^{\mu}_{\nu,+}$ obtained, for instance, by factoring out the Serre subcategory of all modules with no section an irreducible module indexed by an element of $^vI_{\nu,+}^{\mu}$. (As noted above, the quotient can also be taken using (Hom from) a suitable projective module. Also, as noted in the previous paragraph, a finitely generated module version of the parabolic block needs to be used, if only finite length objects are desired in the quotient.)

We can now give \cite{SVV14}'s sufficient condition for the categories on both sides of (\ref{display}) to be Koszul dual, and, at the same time, we give the quite elegant correspondence of labels for irreducibles that achieves this. In fact, the required correspondence
\begin{equation}\label{correspondence} \Psi^{\nu}_\mu:\, I_{\mu,-}^{\nu}\overset\sim\longrightarrow (I_{\nu,+}^{\mu})^{\text{\rm op}}\end{equation}
is given\footnote{The correspondence is not given a name in \cite{SVV14}.} uniformly for fixed $\mu$ and $\nu$ in the form of an anti-isomorphism from the posets $I_{\mu,-}^{\nu}$ to
the poset $I_{\nu,+}^{\mu}$. Let $w_{\mu}$ and $w_{\nu}$ denote the long words in $W_{\mu}$ and $W_{\nu}$, respectively. Then the anti-isomorphism is given by letting letting  $x\in I_{\mu,-}^{\nu}$  correspond to the affine Weyl group element $y=w_{\mu}x^{-1}w_{\nu}$. For any $x$, the element $y$ belongs to $I_{\nu,+}^{\mu}$ if and only if $x$ belongs to $I_{\mu,-}^{\nu}$ \cite[Lem.3.2]{SVV14}. The parameter $w$ in the left side of (\ref{display}) can be any such $x$, in which case $v$ on the right is taken to be the corresponding element $y$.  The anti-isomorphism $\Psi^\nu_\mu$ then restricts to an
anti-isomorphism  $^wI_{\mu,-}^{\nu}\overset\sim\longrightarrow ({^vI}^\mu_{\nu,+})^{\text{\rm op}}$. With $w$ and $v$ chosen in this way, part of the assertion of \cite[Thm.3.6]{SVV14} is that the two sides of (\ref{display}) are Koszul dual to each other, after relabeling the irreducibles (of one side or the other) using this anti-isomorphism. This also makes the underlying poset on one side of
(\ref{display}) the same as the opposite of the poset on the other side---all as expected for Koszul duality.

In some sense this completes our exposition of \cite[Thm.3.6]{SVV14} {\it per\,\,se}, but we have several more remarks to make which clarify the result and extend its scope. First,  the observations of \S5 show that {\it any} finite dimensional algebra whose module category is equivalent to one of the categories in the display inherits the standard Koszulity property, and is even balanced. Second, many more categories for affine Kac-Moody Lie algebras, allowing irreducible modules to have non-integral highest weights, also have these properties, as follows by applying the work of Fiebig  \cite[Thm. 11]{Fiebig06} (cited above). Since Fiebig works with the full (finitely generated version of) the category $\mathbb O$, one needs characterizations of the categories in (\ref{display})  that intrinsically fit his framework. The ``singular" labels of weight orbits carry over with no difficulty, but the ``parabolic" labels need to be treated more carefully. Part (a) of the following proposition, which is a restatement of \cite[Cor.3.3]{SVV14}, gives one way to do this. A somewhat more transparent ``double coset" version is given in part (b). For the latter, we introduce the notion of a {\it regular} double coset $XzY$ for two subgroups $X,Y$ of a given group $Z$. This is a double coset for which the intersection $z^{-1}Xz \cap Y$ is trivial. The definition is independent of the representative element $z$ of the double coset. Note that the ``inverse" double coset $Yz^{-1}X$ of a regular $(X,Y)$ double coset is a regular $(Y,X)$ double coset.

\begin{prop}\label{propabove} (a) The set $I_{\mu,-}^{\nu}$ consists precisely of the elements $x$  with $xw_{\mu}$ in $(I_{\nu}^{\text{\rm max}})^{-1} \cap I_{\mu}^{\text{\rm min}}$. Similarly, $I_{\nu,+}^{\mu}$ consists precisely of the elements $y$ with $yw_{\nu}$ in $(I_{\mu}^{\text{\rm min}})^{-1} \cap I_{\nu}^{\text{\rm max}}$.

(b) Also, $I_{\mu,-}^{\nu}$ consists precisely of maximal length representatives of regular $(W_{\nu}, W_{\mu})$ double cosets. Similarly, $I_{\nu,+}^{\mu}$ consists precisely of minimal length representatives of regular $(W_{\mu}, W_{\nu})$ double cosets.
\end{prop}
\begin{proof} Part (a) is, as noted, an almost verbatim restatement of \cite[Cor.3.3]{SVV14}. We just reduce part (b) to it. Clearly the elements given in part (b) belong to the counterpart sets described in part (a).  Suppose next $xw_{\mu}\in (I_{\nu}^{\text{max}})^{-1} \cap I_{\mu}^{\text{min}}$. Then $w_{\mu}x^{-1}$ is a maximal length element of its left $W_{\nu}$-coset. Equivalently, $xw_{\mu}$ is a maximal length element of its right $W_{\nu}$ coset. This means it has the form $w_{\nu}d$ where the length of $xw_{\mu}$ is the sum of the lengths of $w_{\nu}$ and of $d$. Since also $xw_{\mu}\in I_{\mu}^{\text{min}}$ the length of the element $x=w_{\nu} d w_{\mu}$ is the sum of the lengths of its factors $w_{\nu}$, $d$, and  $w_{\mu}$. This is obviously the maximal possible length for an element of the double coset $W_{\nu}d W_{\mu}$. Applying a well-known theory of Howlett and Kilmoyer (see, e.g., \cite[\S4.3]{DDPW}), this can only occur if $d$ is a distinguished (minimal length) double coset representative, and the intersection $d^{-1}W_{\nu}d \cap W_{\mu}$ is trivial. Thus, $x$ is the maximal length element in a regular $(W_{\nu}, W_{\mu})$  double coset.  Hence, the two characterizations of $I_{\mu,-}^{\nu}$ in parts (a) and (b) agree. Applying the (inverse of) $\Psi^\nu_\mu$ in (\ref{correspondence}), if $y\in I_{\nu,+}^{\mu}$, then $w_{\nu} y^{-1} w_{\mu} $ belongs to $I_{\mu,-}^{\nu}$. So, we now know it is of maximal length in its $(W_{\nu}, W_{\mu})$  double coset, which we also know is regular. The maximal length element of any such regular double coset has the form $w_{\nu} z w_{\mu}$, where $z$ is its distinguished element of minimal length. Comparing the two expressions we have for the maximal length element gives that $y=z^{-1}$ has minimal length in its $W_{\mu}, W_{\nu}$ double coset. The latter double coset is obviously regular, since its inverse is regular. This completes the proof.
\end{proof}

Now, in the framework of \cite{SVV14}, the parabolic blocks $\mathbb O^{\nu}_{\mu,-}$ and $\mathbb O^{\mu}_{\nu,+}$ are precisely the full subcategories, of their ambient category $\mathbb O$ block, of objects whose irreducible sections have highest weights indexed by the affine Weyl group elements in $I_{\mu,-}^{\nu}$ or $I_{\nu,+}^{\mu}$, respectively. Proposition \ref{propabove} shows these indexing elements are characterized inside the affine Weyl group in terms of the subgroups $W_{\mu}, W_{\nu}$. All of this Coxeter group information, together with the signs $-$ or $+$ associated to the level, carry over to the context studied by Fiebig in \cite[p.34 bottom, Thm.11]{Fiebig06} for non-integral weights.

Specifically, working with symmetrizable Kac-Moody Lie algebras, \cite{Fiebig06} considers the block $\Lambda$ of the (non-integral, finitely-generated) category $\mathbb O$ corresponding to a non-integral ``dominant" or anti-dominant weight $\lambda$, not at the critical level. (Weights or blocks below the critical level are called ``negative", and above are called ``positive," just as with the signs $-$ and $+$ we have been using here.) Associated to $\Lambda$ is an ``integral Weyl group" $\mathcal{W}(\Lambda)$ generated by the real root reflections in the ambient Kac-Moody Weyl group that move $\lambda$ by an integral multiple of the reflection's underlying root.
This integral Weyl group is a Coxeter group with generators $\mathcal{S}(\Lambda)$,  and $\Lambda=\mathcal{W}(\Lambda)\cdot \lambda$. The stabilizer stab$(\lambda)$ of $\lambda$ under the dot action is generated by a subset of $\mathcal{S}(\Lambda)$.  Then, \cite[Thm.11]{Fiebig06} says, briefly, that, together with the ``negative" or ``positive" nature of $\lambda$, the Coxeter group $\mathcal{W}(\Lambda)$ with its generating set $\mathcal{S}(\Lambda)$ and subgroup stab$(\lambda)$, is sufficient to determine the block up to a category equivalence.  The proof shows that any two blocks, possibly of different symmetrizable Kac-Moody Lie algebras, but with the same sign and Coxeter group information, are equivalent as $\mathbb C$-categories by an equivalence preserving the Coxeter group indexing of irreducibles.

In particular, it makes sense to define categories $^w\mathbb O^{\nu}_{\mu,-} \,\,\text{and}\,\, ^v\mathbb O^{\mu}_{\nu,+}$ in Fiebig's context, provided $\mathcal{W}(\Lambda)$ is an affine Weyl group $\neq$ stab$(\lambda)$. In that case we can think of  $\nu$ and $\mu$  as proper subsets of $\mathcal{S}(\Lambda)$. (In the original \cite{SVV14} set-up, they are sets of fundamental roots, but these can be identified with their corresponding sets of fundamental reflections.) So, for example, suppose $\lambda$ above is anti-dominant, and $\mu$ is defined by the equality $W_{\mu}=$ stab$(\lambda)$. Let $\nu$ be any proper subset of $\mathcal{S}(\Lambda)$ and $W_{\nu}uW_{\mu}$ any regular double coset, with $w$ its element of maximal length. Then the category $^w\mathbb O^{\nu}_{\mu,-}$ is defined to be the full subcategory of $\Lambda$ formed by those of its objects for which all irreducible sections have highest weight $x\cdot \lambda$ with $x$ the longest element in a regular $(W_{\nu}, W_{\mu})$  double coset and $x\leq w$. Of course, in many common situations (such as the original \cite{SVV14} set-up) the condition on $x$ will imply some version of
$\nu$-dominance, but we need not insist upon it.  We still have the conclusion that the category just defined is standard Koszul and balanced, by combining
\cite[Thm.3.6]{SVV14} and \cite[Thm. 11]{Fiebig06}. A similar definition and conclusion can be made for $^v\mathbb O^{\mu}_{\nu,+}$, though the construction must proceed in two steps, first to get a full subcategory $\mathbb O^{\mu}_{\nu,+}$ of $\Lambda$ (now a ``positive" level block), then passing to a quotient category to get $^v\mathbb O^{\mu}_{\nu,+}$.

One important case where we can be sure that $\mathcal{W}(\Lambda)$ is an affine Weyl group occurs when $\lambda$ is an anti-dominant rational weight of a certain form for an affine Lie algebra, still called here $\bold g$ in keeping with earlier notation in this section. The underlying finite root system is assumed indecomposable, and $\lambda$ is required to have integer coefficients at fundamental weights corresponding to these
roots, but the level $k<-g$ of $\lambda$ may be a rational number, not necessarily an integer. The set of all such weights $\lambda$ is called $\mathcal{C}^-_{\text{rat}}$ in \cite{PS6}, which develops further a theory discussed in  \cite[\S 6]{T} without naming the set of weights involved. While  \cite[\S 6]{T} works primarily with the commutator algebra $\tilde{ \bold g}= [\bold g, \bold g]$, it nevertheless follows from their results that $\mathcal{W}(\Lambda)$, as defined above, is an identifiable affine Weyl group (not always of the same type as $\bold g$).

 In \cite[\S 6]{T}, in preparation for discussing the Kazhdan-Lusztig functor, Tanisaki discusses a category $\mathcal O_k$ of certain $\tilde{\bold g}$ modules with level $k$ as above. Its blocks are naturally equivalent to categories of $\bold g$-modules, as discussed \cite{PS6}, with the restriction functor providing the equivalence. More precisely, each block of ${\mathcal O}_k$ is equivalent to a category $\mathbb O^+(\lambda)$ discussed in \cite[$\S 4,\S 5$]{PS6}, with $\lambda \in \mathcal{C}^-_{\text{rat}}$ of level $k$. When $\lambda$ is integral, we can take it as $o_{\mu,-}$ for a block  $\mathbb O^{\nu}_{\mu,-}$ with $\nu$ the set of fundamental roots in the finite root system, and $\mu$ the set of fundamental roots corresponding to the fundamental reflections in stab$(\lambda)$. If $\lambda$ isn't integral, we can use essentially the same notation, as discussed above. In any case, each resulting block is, by the discussion above, the union of full subcategories, corresponding to the finite poset ideals $^wI^{\nu}_{\mu,-}$, each of which is standard Koszul and balanced. (That is, these full subcategories are each equivalent to finite dimensional module categories for finite dimensional algebras with the standard Koszul and balanced properties.) We remark that, by general highest weight category theory, these properties are inherited by the full subcategories corresponding to any finite poset ideal in  $I^{\nu}_{\mu,-}$.

  We are now ready to deduce the same properties for quantum group blocks, whenever the Kazhdan-Lusztig functor is an equivalence. The latter, as discussed in \cite{T},  is a functor $F_{\ell}:\mathcal{O}_k \to \mathcal{Q}_{\ell}$. The target on the right-hand side is, in effect,  the category of all finite dimensional type 1 modules\footnote{This terminology is not used in \cite{T}, but instead the modules are described as having classical weight space decompositions with appropriate actions of standard generators.} for the Lusztig quantum group over $\mathbb C$ at a primitive $\ell${th} root of unity $\zeta$. The root system associated to the quantum group is the finite root system whose affine version is that of $\bold g$.  The relation between $\ell$ and $k$ is that $k=-(\ell/2D)-g$, where $D\in \{1,2,3\}$ is 1 when the associated finite root system is simply-laced, 2 for types $B,C$, or $F_4$ and 3 for type $G_2$. The cases where $F_{\ell}$ is known to be an equivalence are listed in \cite[p.273]{T}. One typical case occurs when $\zeta^2$ is an $e$th root of unity, with $e$ relatively prime to $D$; an equivalence is listed to occur for any such $e>h$, and for any such $e>0$ in type $A$ or $D_{2n}$, with $e>2$ working for the remaining type $D$ cases.

  When $F_\ell$ is an equivalence, it takes irreducible objects to irreducible objects. The statement of \cite[Thm.7.1]{T} shows that the highest weights of these two corresponding irreducibles have the same coefficients at fundamental weights associated to the underlying finite root system. This makes it tempting to use the dominance order for the highest weight order in both domain and range of $F_\ell$, but this actually leads to complications on the left.  Another ``natural" choice is to use the intrinsic order arising from standard modules (which, in both domain and range, have the same composition factors as costandard modules with the same indexing). This works, but is a little abstract, especially on the left. There, we have been using, in our discussion above,  block-by-block Bruhat orders. All of these issues are thoroughly discussed in \cite{PS6}. See especially  \cite[Appendix I]{PS6}, which proves, among  other ordering results, that the traditional ``up arrow" order $\uparrow$ on the right is, within any given block, equivalent to a Bruhat order on the right--actually two of them, one using a base weight in the lowest dominant closed alcove, and the other starting from its anti-dominant counterpart. The latter gives an easy way to match up weights with and partial orders on both domain and range, working block-by-block. Then, as required, we can translate back to the order $\uparrow$ as needed.

  As a corollary of all of the above discussion, we have the following result, ultimately a corollary of \cite[Thm.3.6]{SVV14}, though also depending, as indicated, on work of other authors and our synthesis here. For an abelian category $\sC$ and an indexing by a set $\Gamma$ of some of its simple objects, let $\sC[\Gamma]$ denote the full subcategory of $\sC$ consisting of all objects whose simple sections may all be found among those indexed by $\Gamma$.

  \begin{cor}\label{generalSVV} Assume the Kazhdan-Lusztig functor $F_\ell$ above is an equivalence, for a given $\bold g$ and $\ell$. Let $\Gamma$ be any finite poset, using either the dominance  or $\uparrow$ order, in the dominant weights for the  finite root system associated to the quantum group module category $\mathcal{Q}_{\ell}$ above. Then $\mathcal{Q}_{\ell}[\Gamma]$ is a highest weight category with respect to the given partial order on $\Gamma$, equivalent to the module category of a finite dimensional quasi-hereditary algebra $B'$ over $\mathbb C$.  The algebra $B'$, or any algebra Morita equivalent to it, is standard Koszul and balanced. In particular, $B'\cong \gr B'$.
  \end{cor}

  The proof has already been given in the preceding discussion. Note that any poset ideal of dominant weights with respect to the dominance order is also a poset ideal with respect to the order $\uparrow$.

  \begin{rem}\label{q-Schur}(a) Note that we work over  $\mathbb C$ in the above corollary and throughout this section, whereas in similar situations in $\S 1$ we were working over the field $\mathbb Q(\zeta)$. The algebra $B'$ above may be obtained by base change from a corresponding algebra $B$ over $\mathbb Q(\zeta)$. We don't know if $B$ shares the standard Koszul and balanced properties that $B'$ has, but $\gr B$ does share them, from the last assertion of the corollary.

  (b) In \cite[Cor.6.6]{SVV14} it is proved that the $q$-Schur algebra is Morita equivalent to an algebra that is (standard Koszul and)
balanced. Context indicates $q$ is a root of unity (arbitrary), and the underlying field is $\mathbb C$. The previous section shows, then, that the $q$-Schur algebra itself,
over $\mathbb C$, is standard Koszul and balanced. All the conclusions of the corollary apply to it, or to any Morita equivalent algebra. The $q$-Schur algebra may, of course, be obtained by base-change from an algebra over $\mathbb Q(\zeta)$, so the remark above applies, as well.
\end{rem}
\section{\bf Some complements to \cite[Conj. II]{PS14}}
We have already mentioned one of the three main conjectures in \cite{PS14}, namely Conjecture \ref{conjI} concerning the ubiquity of standard Q-Koszul algebras in modular representation theory. The next conjecture appears not to involve Q-Koszul algebras or forced gradings at all, but, as we shall see,
their roles are hidden. The conjecture says simply that certain Ext group dimensions (for the algebraic
group $G$) can be computed from corresponding dimensions
for a quantum group at a $\ell(p)$th root of unity $\zeta$. (Recall that $\ell(p)=p$ for each odd prime, and  $=4$ when $p=2$.) The Ext groups for $A$ below are the same as the corresponding Ext groups for $G$, and the Ext groups for $B$ are the same as for the quantum enveloping algebra $U_{\zeta}$. (The algebraic groups case is well-known, and the quantum case may be found in \cite{DS}.)

\medskip\noindent
\begin{conj}\label{conjII} Continue to assume the hypotheses and notation of Conjecture \ref{conjI}.  (So, in particular, $p$ is KL-good, $\Gamma$ is a finite poset ideal of dominant weights associated to the root system of a semisimple algebraic group $G$, and $A:=\Dist(G)_\Gamma$ is a finite-dimensional quasi-hereditary algebra.) Using the same poset $\Gamma$, let $B=U_{\zeta,\Gamma}$ be the corresponding algebra for the quantum enveloping algebra $U_{\zeta}$ associated to that root system
at a primitive $\ell(p)$th root of unity $\zeta$ (as in Theorem \ref{KoszulforgeneralizedqSchur}).  Then
$$\begin{cases}(1)\quad\dim\Ext^n_{B}(\Delta_{\zeta}(\lambda),L_{\zeta}(\mu))&=\dim\Ext^n_A(\Delta(\lambda),\rnabla(\mu)),\\
(2)\quad \dim\Ext^n_{B}(L_{\zeta}(\lambda),\nabla_{\zeta}(\mu)) & =\dim\Ext^n_A(\rDelta(\lambda),\nabla(\mu))\\
(3)\quad \dim\Ext^n_{B}(L_{\zeta}(\lambda),L_{\zeta}(\mu)) & =\dim\Ext^n_A(\rDelta(\lambda),\rnabla(\mu))\end{cases},
\quad \forall \lambda,\mu\in\Lambda, \forall n\geq 0$$
\end{conj}

When $p>h$ and the Lusztig
character formula holds for $G$ and $p$-restricted dominant weights, this conjecture is proved for $p$-regular weights in \cite[Thm. 5.4]{CPS09}.
Some interesting cases can be proved assuming only $p>h$, along the lines of Theorem \ref{goodfiltrationtheorem}(c); see the discussion in \cite{PS14}. In the latter paper, we showed that, in cases where Conjecture \ref{conjI} is true,
Conjecture \ref{conjII} above reduces to two further conjectures, one of these can now be stated below as a theorem. It follows from Corollary \ref{generalSVV}, proved in the previous section. (Note that the Ext group dimensions in the theorem remain the same after base change to $\mathbb C$.)

\begin{thm}\label{conjIIb}Under the hypothesis of Conjecture \ref{conjII}, we have
$$\begin{cases}(1)\quad\dim\Ext^n_{B}(\Delta_{\zeta}(\lambda),L_{\zeta}(\mu))&=\dim\Ext^n_{\gr B}(\gr \Delta_{\zeta}(\lambda),L_{\zeta}(\mu))\\
(2)\quad\dim\Ext^n_{B}(L_{\zeta}(\lambda),\nabla_{\zeta}(\mu))&=\dim\Ext^n_{\gr B}(L_{\zeta}(\lambda),\gr^\diamond\nabla_{\zeta}(\mu)),\\
(3)\quad\dim\Ext^n_{B}(L_{\zeta}(\lambda),L_{\zeta}(\mu))&=\dim\Ext^n_{\gr B}(L_{\zeta}(\lambda),L_{\zeta}(\mu))\end{cases}$$
for all $\lambda,\mu\in \Gamma$ and all $n\in\mathbb N$. Also, $\gr B$ is a standard Koszul algebra.
\end{thm}

Since this result holds, the argument in \cite{PS14} shows Conjecture \ref{conjII} reduces, whenever Conjecture \ref{conjI} is true, to the conjecture below (called Conjecture IIa in \cite{PS14}). The argument for the reduction is just a base-change carried out in the graded case, made extremely easy to handle because of the Q-Koszul property.

\smallskip
 \begin{conj}\label{conjIIa} Under the hypothesis of Conjecture \ref{conjII}, we have
$$\begin{cases}(1)\quad \dim\Ext^n_A(\Delta(\lambda),\rnabla(\mu))&=\dim\Ext^n_{\wgr A}(\wgr \Delta(\lambda),\rnabla(\mu))\\
(2)\quad
\dim\Ext^n_A(\rDelta(\mu),\nabla(\lambda)) & =\dim\Ext^n_{\wgr A}(\rDelta(\mu),\wgr^\diamond\nabla(\lambda))\\
(3)\quad\dim\Ext^n_A(\rDelta(\mu),\rnabla(\lambda)) & =\dim\Ext^n_{\wgr A}(\rDelta(\mu),\rnabla(\lambda)).\end{cases}$$
for all $\lambda,\mu\in \Gamma$.
\end{conj}

In part (2) above, $\wgr^\diamond\nabla(\lambda)$ denotes the costandard module corresponding to $\lambda$
in the highest weight category $\wgr A$-mod. It has a natural graded structure, concentrated in non-positive grades, with
$\rnabla(\lambda)$ its grade 0 term.   Under the assumptions that $\lambda,\mu$ are $p$-regular and $p\geq 2h-2$ is
an odd prime, the conjecture follows from Theorem \ref{comparisonofcohomology}. In particular,
both Conjecture \ref{conjI} and Conjecture \ref{conjII} are true for $p$ sufficiently large without any regularity assumptions. There is a third conjecture, called Conjecture III in \cite{PS14}, which provides explicit formulas in terms of Kazhdan-Lusztig polynomials for the dimensions of the quantum
Ext groups appearing in Conjecture \ref{conjII}.
Although Conjecture III only gives a formula for the quantum Ext-dimension in part (1) of Conjecture \ref{conjIIa}, the
dimension in part (2) may be obtained by a duality. Then part (3) may be computed as in \cite[\S4]{PS1}, using
Theorem \ref{conjIIb}.  In the $p$-regular weight case with $p>h$ such formulas are already known in all three cases.

Finally, it is an interesting question as to when $A$ itself has a positive grading such that $A\cong\wgr A$.  Geometric
conjectures of Achar-Riche \cite{AR} suggest this could be true at least for modestly large $p$. Of course, in that case,
Conjucture \ref{conjIIa} is an easy consequence.

\end{document}